\documentclass[12pt]{article}
\usepackage[centertags]{amsmath}
\usepackage{amsfonts}
\usepackage{amssymb}
\usepackage{latexsym}
\usepackage{amsthm}
\usepackage{newlfont}
\usepackage{graphicx}
\usepackage{listings}
\usepackage{booktabs}
\usepackage{abstract}
\date{}

\newlength{\defbaselineskip}
\newcommand{\setlinespacing}[1]%
           {\setlength{\baselineskip}{#1 \defbaselineskip}}

\newcommand{\bR}{{\mathbb{R}}}

\newcommand{\N}{{\mathbb{N}}}

\newcommand{\actaqed}{\hfill $\actabox$}
{\medskip\noindent \textit{Proof of #1. }}%
{\actaqed \medskip}

\def\D{{\mathcal D}}
\def\Di{{\mathcal D}}

\def\C{{\mathcal C}}

\def \Tr{\mathcal T}

\def \U{\mathcal U}

\def \cM{\mathcal M}
\def\R{{\mathbb R}}
\def\Z{\mathbb Z}

\def \T{\mathbb T}
\def\bP{\mathbb P}
\def\bE{\mathbb E}
\def \<{\langle}
\def\>{\rangle}

\def \La{\Lambda}
\def \ep{\epsilon}
\def \e{\varepsilon}
\def \de{\delta}

\def \ff{\varphi}
\def\al{\alpha}
\def\bt{\beta}
\def\ga{\gamma}
\def\la{\lambda}

\def \sp{\operatorname{span}}

\def\bx{\mathbf x}
\def\by{\mathbf y}
\def\bz{\mathbf z}
\def\bk{\mathbf k}

\def\bn{\mathbf n}
\def\bs{\mathbf s}
\def\bN{\mathbf N}
\def\bW{\mathbf W}

\def\bt{\beta}
\def\mb{{\bar m}}

\newtheorem{Theorem}{Theorem}[section]
\newtheorem{Lemma}{Lemma}[section]

\newtheorem{Proposition}{Proposition}[section]
\newtheorem{Remark}{Remark}[section]

\numberwithin{equation}{section}

\newcommand{\be}{\begin{equation}}
\newcommand{\ee}{\end{equation}}

\begin{document}

\title{The Marcinkiewicz-type discretization theorems}
\author{  V.N. Temlyakov\thanks{University of South Carolina and Steklov Institute of Mathematics.  }}
\maketitle
\begin{abstract}
{The paper is devoted to discretization of integral norms of functions from
a given finite dimensional subspace. This problem is very important in applications but there is no systematic study of it. We present here a new technique, which works well for discretization of the integral norm. It is a combination of probabilistic technique, based on chaining, with results on the entropy numbers in the uniform norm. }
\end{abstract}

\section{Introduction} 
\label{I} 

Discretization is 
a very important step in making a continuous problem computationally feasible. The problem of construction of good sets of points in a multidimensional domain is a fundamental problem of mathematics and computational mathematics. 
A prominent  example of classical discretization problem is a problem of metric entropy (covering numbers, entropy numbers). 
Bounds for the entropy numbers of function classes are important by themselves and also have important connections to other fundamental problems (see, for instance, \cite{Tbook}, Ch.3 and \cite{DTU}, Ch.6). 
Another prominent example of a discretization problem is the problem of numerical integration.  Numerical integration in the mixed smoothness classes requires deep number theoretical results for constructing optimal (in the sense of order) cubature formulas (see, for instance, \cite{DTU}, Ch.8). A typical approach to solving a continuous problem numerically -- the Galerkin method -- 
suggests to look for an approximate solution from a given finite dimensional subspace. A standard way to measure an error of approximation is an appropriate $L_q$ norm, $1\le q\le\infty$. Thus, the problem of   discretization of the $L_q$ norms of functions from a given finite dimensional subspace arises in a very natural way.

The main goal of this paper is to study the discretization problem for a finite dimensional subspace $X_N$ of a Banach space $X$. We are interested in 
discretizing the $L_q$, $1\le q\le \infty$, norm of elements of $X_N$. We call 
such results the Marcinkiewicz-type discretization theorems. There are different 
settings and different ingredients, which play important role in this problem. 
We now discuss these issues. 

{\bf Marcinkiewicz problem.} Let $\Omega$ be a compact subset of $\R^d$ with the probability measure $\mu$. We say that a linear subspace $X_N$ of the $L_q(\Omega)$, $1\le q < \infty$, admits the Marcinkiewicz-type discretization theorem with parameters $m$ and $q$ if there exist a set $\{\xi^\nu \in \Omega, \nu=1,\dots,m\}$ and two positive constants $C_j(d,q)$, $j=1,2$, such that for any $f\in X_N$ we have
\be\label{1.1}
C_1(d,q)\|f\|_q^q \le \frac{1}{m} \sum_{\nu=1}^m |f(\xi^\nu)|^q \le C_2(d,q)\|f\|_q^q.
\ee
In the case $q=\infty$ we define $L_\infty$ as the space of continuous on $\Omega$ functions and ask for 
\be\label{1.2}
C_1(d)\|f\|_\infty \le \max_{1\le\nu\le m} |f(\xi^\nu)| \le  \|f\|_\infty.
\ee
We will also use a brief way to express the above property: the $\cM(m,q)$ theorem holds for  a subspace $X_N$ or $X_N \in \cM(m,q)$. 

{\bf Numerical integration problem.} In the case $1\le q<\infty$ the above problem can be reformulated as a problem 
on numerical integration of special classes of functions. Define a class
$|X_N|^q := \{|f|^q: f\in X_N, \|f\|_q \le 1\}$ and consider the numerical integration problem: for a given $\e>0$ find $m = m(N,q,\e)$ such that 
\be\label{1.3}
\inf_{\xi^1,\dots,\xi^m}\sup_{f\in X_N,\|f\|_q\le 1} \left|\frac{1}{m}\sum_{\nu=1}^m |f(\xi^\nu)|^q -\|f\|_q^q\right| \le \e.
\ee
In (\ref{1.3}) we limit our search for good numerical integration methods to 
cubature formulas with equal weights $1/m$. This special kind of cubature formulas is called the Quasi-Monte Carlo methods. In numerical integration 
general cubature formulas (with weights) are also very important. In this case 
the above problem (\ref{1.3}) is reformulated as follows
\be\label{1.4}
\inf_{\xi^1,\dots,\xi^m;\la_1,\dots,\la_m}\sup_{f\in X_N,\|f\|_q\le 1} \left|\sum_{\nu=1}^m\la_\nu |f(\xi^\nu)|^q -\|f\|_q^q\right| \le \e.
\ee
Thus, in this case we are optimizing both over the knots $\xi^1,\dots,\xi^m$ and over the weights $\la_1,\dots,\la_m$. 

{\bf Marcinkiewicz problem with weights.} The above remark on numerical integration encourages us to consider the following variant of the Marcinkiewicz problem. We say that a linear subspace $X_N$ of the $L_q(\Omega)$, $1\le q < \infty$, admits the weighted Marcinkiewicz-type discretization theorem with parameters $m$ and $q$ if there exist a set of knots $\{\xi^\nu \in \Omega\}$, a set of weights $\{\la_\nu\}$, $\nu=1,\dots,m$, and two positive constants $C_j(d,q)$, $j=1,2$, such that for any $f\in X_N$ we have
\be\label{1.5}
C_1(d,q)\|f\|_q^q \le  \sum_{\nu=1}^m \la_\nu |f(\xi^\nu)|^q \le C_2(d,q)\|f\|_q^q.
\ee
Then we also say that the $\cM^w(m,q)$ theorem holds for  a subspace $X_N$ or $X_N \in \cM^w(m,q)$. 
Obviously, $X_N\in \cM(m,q)$ implies that $X_N\in \cM^w(m,q)$. 

{\bf Marcinkiewicz problem with $\e$.} We write $X_N\in \cM(m,q,\e)$ if (\ref{1.1}) holds with $C_1(d,q)=1-\e$ and $C_2(d,q)=1+\e$.  Respectively, 
we write $X_N\in \cM^w(m,q,\e)$ if (\ref{1.5}) holds with $C_1(d,q)=1-\e$ and $C_2(d,q)=1+\e$. We note that the most powerful results are for $\cM(m,q,0)$, 
when the $L_q$ norm of $f\in X_N$ is discretized exactly by the formula with equal weights $1/m$. 

In this paper we mostly concentrate on the Marcinkiewicz problem and on its variant with $\e$. Our main results are for $q=1$. We now give some general remarks for the case $q=2$, which illustrate the problem. We discuss the case $q=2$ in more detail in Section \ref{L2}. We describe the properties of the subspace $X_N$ in terms of a system $\U_N:=\{u_i\}_{i=1}^N$ of functions such that
$X_N = \sp\{u_i, i=1,\dots,N\}$. In the case $X_N \subset L_2$ we assume that 
the system is orthonormal on $\Omega$ with respect to measure $\mu$. In the case of real functions   we associate with $x\in\Omega$ the matrix $G(x) := [u_i(x)u_j(x)]_{i,j=1}^N$. Clearly, $G(x)$ is a symmetric positive semi-definite matrix of rank $1$. 
It is easy to see that for a set of points $\xi^k\in \Omega$, $k=1,\dots,m$, and $f=\sum_{i=1}^N b_iu_i$ we have
$$
 \sum_{k=1}^m\la_k f(\xi^k)^2 - \int_\Omega f(x)^2 d\mu = {\mathbf b}^T\left(\sum_{k=1}^m \la_k G(\xi^k)-I\right){\mathbf b},
$$
where ${\mathbf b} = (b_1,\dots,b_N)^T$ is the column vector and $I$ is the identity matrix. Therefore, 
the $\cM^w(m,2)$ problem is closely connected with a problem of approximation (representation) of the identity matrix $I$ by an $m$-term approximant with respect to the system $\{G(x)\}_{x\in\Omega}$. It is easy to understand that under our assumptions on the system $\U_N$ there exist a set of knots $\{\xi^k\}_{k=1}^m$ and a set of weights $\{\la_k\}_{k=1}^m$, with $m\le N^2$ such that
$$
I = \sum_{k=1}^m \la_k G(\xi^k)
$$
and, therefore, we have for any $X_N \subset L_2$ that
\be\label{1.6}
X_N \in \cM^w(N^2,2,0).
\ee
However, we do not know a characterization of those $X_N$ for which 
$X_N \in \cM(N^2,2,0)$. 

In the above formulations of the problems we only ask about existence of 
either good $\{\xi^\nu\}$ or good $\{\xi^\nu,\la_\nu\}$. Certainly, it is important to 
have either explicit constructions of good $\{\xi^\nu\}$ ($\{\xi^\nu,\la_\nu\}$) or 
deterministic ways to construct good $\{\xi^\nu\}$ ($\{\xi^\nu,\la_\nu\}$). Thus, the 
Marcinkiewicz-type problem can be split into the following four problems: under some assumptions on $X_N$

(I) Find a condition on $m$ for $X_N \in \cM(m,q)$;

(II) Find a condition on $m$ for $X_N \in \cM^w(m,q)$;

(III) Find a condition on $m$ such that there exists a deterministic construction 
of $\{\xi^\nu\}_{\nu=1}^m$ satisfying (\ref{1.1}) for all $f\in X_N$;

(IV) Find a condition on $m$ such that there exists a deterministic construction 
of $\{\xi^\nu,\la_\nu\}_{\nu=1}^m$ satisfying (\ref{1.5}) for all $f\in X_N$.

The main results of this paper address the problem (I) in the case $q=1$. 
Our method is probabilistic. 

We impose the following assumptions on the system $\{u_i\}_{i=1}^N$ of real functions. 

{\bf A.} There exist $\alpha>0$, $\beta$, and $K_1$ such that for all $i\in [1,N]$ we have
\be\label{4.1i}
|u_i(\bx)-u_i(\by)| \le K_1N^\beta\|\bx-\by\|_\infty^\alpha,\quad \bx,\by \in \Omega.
\ee

{\bf B.} There exists a constant $K_2$ such that $\|u_i\|_\infty^2 \le K_2$, $i=1,\dots,N$.  
 
{\bf C.} Denote $X_N:= \sp(u_1,\dots,u_N)$. There exist two constants $K_3$ and $K_4$ such that the following Nikol'skii-type inequality holds for all $f\in X_N$
\be\label{4.3i} 
\|f\|_\infty \le K_3N^{K_4/p}\|f\|_p,\quad p\in [2,\infty).
\ee

The main result of this paper is the following theorem (see Theorem \ref{T4.9}).

\begin{Theorem}\label{T1.1} Suppose that a real orthonormal system $\{u_i\}_{i=1}^N$ satisfies conditions {\bf A}, {\bf B}, and {\bf C}. Then there exists a set of $m \le C_1N(\log N)^{7/2}$ points $\xi^j\in \Omega$, $j=1,\dots,m$, $C_1=C(d,K_1,K_2,K_3,K_4,\Omega,\alpha,\beta)$, such that for any $f\in X_N$ 
we have
$$
\frac{1}{2}\|f\|_1 \le \frac{1}{m}\sum_{j=1}^m |f(\xi^j)| \le \frac{3}{2}\|f\|_1.
$$
\end{Theorem}

An important particular case for application of Theorem \ref{T1.1} is the case, when $X_N$ is a subspace of trigonometric polynomials. For a finite $Q\subset \Z^d$ denote
$$
\Tr(Q) := \{f: f(\bx) = \sum_{\bk\in Q}c_\bk e^{i(\bk,\bx)}\}.
$$
The hyperbolic cross polynomials $\Tr(Q_n)$ are of special interest (see, for instance, \cite{DTU}): let $\bs \in \Z^d_+$
$$
Q_n := \cup_{\|\bs\|_1 \le n} \rho(\bs),
$$
  and
$$
\rho(\bs):= \{\bk=(k_1,\dots,k_d)\in \Z^d_+ : [2^{s_j-1}]\le |k_j|<2^{s_j},\quad j=1,\dots,d\}
$$
where $[a]$ denotes the integer part of a number $a$.

The following two theorems were proved in \cite{VT160}.

\begin{Theorem}\label{T3.1} Let $d=2$. For any $n\in \N$ there exists a set of $m \le C_1|Q_n|n^{7/2}$ points $\xi^j\in \T^2$, $j=1,\dots,m$ such that for any $f\in \Tr(Q_n)$ 
we have
$$
C_2\|f\|_1 \le \frac{1}{m}\sum_{j=1}^m |f(\xi^j)| \le C_3\|f\|_1.
$$
\end{Theorem}

\begin{Theorem}\label{T3.2} For any $d\in \N$ and $n\in \N$ there exists a set of \newline $m \le C_1(d)|Q_n|n^{d/2+3}$ points $\xi^j\in \T^d$, $j=1,\dots,m$ such that for any $f\in \Tr(Q_n)$ 
we have
$$
C_2\|f\|_1 \le \frac{1}{m}\sum_{j=1}^m |f(\xi^j)| \le C_3\|f\|_1.
$$
\end{Theorem}

Theorem \ref{T3.1} addresses the case $d=2$ and Theorem \ref{T3.2} extends 
Theorem \ref{T3.1} to the case of all $d$. We point out that for $d=2$ 
Theorem \ref{T3.2} is weaker than Theorem \ref{T3.1}. Theorem \ref{T1.1} gives 
Theorem \ref{T3.1} and improves Theorem \ref{T3.2} by replacing an extra factor $n^{d/2+3}$ by $n^{7/2}$ in the bound for $m$. The technique for proving 
Theorem \ref{T1.1} presented in this paper is a development of technique from \cite{VT160}. It is a combination of probabilistic technique, based on chaining, with results on the entropy numbers. We present this technique in the following way. In Section \ref{entropy} we discuss new elements of a method, which gives good upper bounds for the entropy  numbers of the unit $L_1$ ball $\Tr(Q)_1$ in the $L_\infty$ norm. In Section \ref{M} we present results on the 
Marcinkiewicz-type theorems for the trigonometric polynomials. In Section \ref{MX} we show how the technique developed in Section \ref{M} for the trigonometric polynomials can be generalized for subspaces $X_N$ satisfying conditions {\bf A}, {\bf B}, and {\bf C}. The main results of the paper are in Sections \ref{entropy} -- \ref{MX}. They are about discretization theorems in $L_1$. In Section \ref{L2} we give some comments on the discretization theorems in $L_2$. This case (the $L_2$ case) is much better understood than the $L_1$ case and it has nice connections to recent strong results on submatrices of orthogonal matrices and on random matrices.

\section{The entropy numbers of $\Tr(Q)_1$}
\label{entropy}

We begin with the definition of the entropy numbers.
  Let $X$ be a Banach space and let $B_X$ denote the unit ball of $X$ with the center at $0$. Denote by $B_X(y,r)$ a ball with center $y$ and radius $r$: $\{x\in X:\|x-y\|\le r\}$. For a compact set $A$ and a positive number $\e$ we define the covering number $N_\e(A)$
 as follows
$$
N_\e(A) := N_\e(A,X) 
:=\min \{n : \exists y^1,\dots,y^n, y^j\in A :A\subseteq \cup_{j=1}^n B_X(y^j,\e)\}.
$$
It is convenient to consider along with the entropy $H_\e(A,X):= \log_2 N_\e(A,X)$ the entropy numbers $\e_k(A,X)$:
$$
\e_k(A,X)  :=\inf \{\e : \exists y^1,\dots ,y^{2^k} \in A : A \subseteq \cup_{j=1}
^{2^k} B_X(y^j,\e)\}.
$$
In our definition of $N_\e(A)$ and $\e_k(A,X)$ we require $y^j\in A$. In a standard definition of $N_\e(A)$ and $\e_k(A,X)$ this restriction is not imposed. 
However, it is well known (see \cite{Tbook}, p.208) that these characteristics may differ at most by a factor $2$.

We use the technique developed in \cite{VT156}, which is based on the following two steps strategy. At the first step we obtain bounds of the best $m$-term approximations with respect to a dictionary. At the second step we use general inequalities relating the entropy numbers to the best $m$-term approximations. We begin the detailed discussion with the second step of the above strategy.
Let $\D=\{g_j\}_{j=1}^N$ be a system of elements of cardinality $|\D|=N$ in a Banach space $X$. Consider best $m$-term approximations of $f$ with respect to $\D$
$$
\sigma_m(f,\D)_X:= \inf_{\{c_j\};\Lambda:|\Lambda|=m}\|f-\sum_{j\in \Lambda}c_jg_j\|.
$$
For a function class $F$ set
$$
\sigma_m(F,\D)_X:=\sup_{f\in F}\sigma_m(f,\D)_X.
$$
The following results are from \cite{VT138}.
\begin{Theorem}\label{T2.1} Let a compact $F\subset X$ be such that there exists a   system $\D$, $|\D|=N$, and  a number $r>0$ such that 
$$
  \sigma_m(F,\D)_X \le m^{-r},\quad m\le N.
$$
Then for $k\le N$
\begin{equation}\label{2.1}
\e_k(F,X) \le C(r) \left(\frac{\log(2N/k)}{k}\right)^r.
\end{equation}
\end{Theorem}
\begin{Remark}\label{R2.1} Suppose that a compact $F$ from Theorem \ref{T2.1} belongs to an $N$-dimensional subspace $X_N:=\sp(\D)$. Then in addition to (\ref{2.1}) we have 
 for $k\ge N$
\begin{equation}\label{2.2}
\e_k(F,X) \le C(r)N^{-r}2^{-k/(2N)}.
\end{equation}
\end{Remark}
We point out that Remark \ref{R2.1} is formulated for a complex Banach space $X$. In the case of real Banach space $X$ we have $2^{-k/N}$ instead of $2^{-k/(2N)}$ in (\ref{2.2}).

We begin with the best $m$-term approximation of elements of $\Tr(Q)_1:=\{f\in\Tr(Q):\|f\|_1\le 1\}$ in $L_2$ with respect to a special dictionary $\D^1:=\D^1(Q)$ associated with $Q$. 
Denote
$$
\D_Q(\bx) := \sum_{\bk\in Q} e^{i(\bk,\bx)},\quad w_Q := |Q|^{-1/2}\D_Q.
$$
Then $\|w_Q\|_2 =1$. Consider the dictionary
$$
\D^1:= \D^1(Q):= \{w_Q(\bx-\by)\}_{\by\in\T^d}.
$$
For a dictionary $\D$ in a Hilbert space $H$ with an inner product $\<\cdot,\cdot\>$ denote by $A_1(\D)$ the closure 
of the convex hull of the dictionary $\D$. In the case of complex Hilbert space define the symmetrized dictionary $\D^s := \{e^{i\theta} g: g\in \D, \theta\in [0,2\pi]\}$.
We use the Weak Orthogonal Greedy Algorithm (Weak Orthogonal Matching Pursuit) for $m$-term approximation. We remind the corresponding definition and formulate the know result, which we will use. 

 {\bf Weak Orthogonal Greedy Algorithm (WOGA).} Let $t\in (0,1]$ be a weakness parameter. We define
 $f^{o,t}_0 :=f$. Then for each $m\ge 1$ we inductively define:

(1) $\varphi^{o,t}_m \in \D$ is any element satisfying
$$
|\langle f^{o,t}_{m-1},\varphi^{o,t}_m\rangle | \ge t
\sup_{g\in \D} |\langle f^{o,t}_{m-1},g\rangle |.
$$

(2) Let $H_m^t := \sp (\varphi_1^{o,t},\dots,\varphi^{o,t}_m)$ and let
$P_{H_m^t}(f)$ denote an operator of orthogonal projection onto $H_m^t$.
Define
$$
G_m^{o,t}(f,\D) := P_{H_m^t}(f).
$$

(3) Define the residual after $m$th iteration of the algorithm
$$
f^{o,t}_m := f-G_m^{o,t}(f,\D).
$$

In the case $t=1$ the   WOGA is called the Orthogonal
Greedy Algorithm (OGA).  The following theorem is from \cite{T1} (see also \cite{Tbook}, Ch.2).

\begin{Theorem}\label{T2.2}
Let $\D$ be an arbitrary dictionary in $H$. Then for each $f \in
A_1(\D^s)$ we have
\begin{equation}\label{2.3}
\|f-G^{o,t}_m(f,\D)\| \le (1+mt^2)^{-1/2}.
\end{equation}
\end{Theorem}

We now prove the following assertion.

\begin{Theorem}\label{T2.3} For any finite $Q\subset \Z^d$ we have
$$
\sigma_m(\Tr(Q)_1,\D^1(Q))_2 \le (|Q|/m)^{-1/2}.
$$
\end{Theorem}
\begin{proof} Each $f\in \Tr(Q)_1$ has a representation
\be\label{2.4}
f(\bx) = (2\pi)^{-d} \int_{\T^d} f(\by)\D_Q(\bx-\by)d\by = |Q|^{1/2}(2\pi)^{-d} \int_{\T^d} f(\by)w_Q(\bx-\by)d\by.
\ee
It follows from $\|f\|_1 \le 1$ and (\ref{2.4}) that $f|Q|^{-1/2} \in A_1((\D^1)^s)$. 
Therefore, by Theorem \ref{T2.2} we get the required bound.
\end{proof} 

Dictionary $\D^1(Q)$ is an infinite dictionary. In our further applications we would like to have a finite dictionary. Here we consider $Q\subset \Pi(\bN)$ with $\bN=(2^n,\dots,2^n)$, where $\Pi(\bN):=[-N_1,N_2]\times\cdots\times[-N_d,N_d]$, $\bN=(N_1,\dots,N_d)$. We denote
\begin{align*}
P(\mathbf N) := \bigl\{\mathbf n = (n_1 ,\dots,n_d),&\qquad n_j\ -\
\text{ are nonnegative integers},\\
&0\le n_j\le 2N_j  ,\qquad j = 1,\dots,d \bigr\},
\end{align*}
and set
$$
\bx^{\mathbf n}:=\left(\frac{2\pi n_1}{2N_1+1},\dots,\frac{2\pi n_d}
{2N_d+1}\right),\qquad \mathbf n\in P(\mathbf N) .
$$
Then for any $t\in \Tr(\Pi(\mathbf N))$ (see \cite{Z}, Ch.10)
\be\label{2.5}
\vartheta(\mathbf N)^{-1}\sum_{\mathbf n\in P(\mathbf N)}
\bigl|t(\bx^{\mathbf n})\bigr|\le C(d)\|t\|_1,   
\ee
where $\vartheta(\mathbf N) := \prod_{j=1}^d (2N_j  + 1)=\dim\Tr(\Pi(\bN))$.  
Specify $\bN:=(2^n,\dots,2^n)$ and define 
$$
\D^2:=\D^2(Q):= \{w_Q(\bx-\bx^\bn)\}_{\bn\in P(\bN)}.
$$
Then, clearly, $|\D^2(Q)|=\vartheta(\bN) = (2^{n+1}+1)^d$. Also, it is well known
that for $f\in \Tr(\Pi(\bN))$ one has 
\be\label{2.6}
f(\bx) = \vartheta(\mathbf N)^{-1}\sum_{\mathbf n\in P(\mathbf N)}
f(\bx^{\mathbf n})\D_{\Pi(\bN)}(\bx-\bx^\bn)
\ee
and, therefore, for $f\in \Tr(Q)$, $Q\subset \Pi(\bN)$
\be\label{2.6Q}
f(\bx) = \vartheta(\mathbf N)^{-1}\sum_{\mathbf n\in P(\mathbf N)}
f(\bx^{\mathbf n})\D_Q(\bx-\bx^\bn).
\ee
In particular, (\ref{2.5}) and (\ref{2.6Q}) imply that there exists $C(d)>0$ such that for every $f\in\Tr(Q)_1$ we have $C(d)^{-1}|Q|^{-1/2} f\in A_1((\D^2)^s)$. 
Therefore, we have the following version of Theorem \ref{T2.3}.

\begin{Theorem}\label{T2.4} For any $Q\subset \Pi(\bN)$ with $\bN=(2^n,\dots,2^n)$ we have
$$
\sigma_m(\Tr(Q)_1,\D^2(Q))_2 \le C(d)(|Q|/m)^{-1/2}
$$
and $|\D^2(Q)|\le C'(d)2^{nd}$.
\end{Theorem}

Theorems \ref{T2.3} and \ref{T2.4} provide bounds for the best $m$-term approximation of elements of $\Tr(Q)_1$ in the $L_2$ norm. For applications in 
the Marcinkiewicz discretization theorem we need bounds for the entropy numbers in the $L_\infty$ norm. As we explained above we derive appropriate bounds for the entropy numbers from the corresponding bounds on the best $m$-term approximations with the help of Theorem \ref{T2.1}. Thus we need bounds on the best $m$-term approximations in the $L_\infty$ norm. We proceed in the same way as in \cite{VT156} and use the following dictionary
$$
\D^\Tr:=\D^\Tr(Q):= \{e^{i(\bk,\bx)}: \bk\in   Q\}.
$$
In order to obtain the bounds in the $L_\infty$ norm we use the following theorem from \cite{VT156}, which in turn is a corollary of the corresponding result from \cite{VT150}.

\begin{Theorem}\label{T2.5} Let $\Lambda\subset \Pi(\bN)$ with $N_j=2^n$, $j=1,\dots,d$. There exist constructive greedy-type approximation methods $G^\infty_m(\cdot)$, which provide $m$-term polynomials with respect to $\Tr^d$ with the following properties: 
\newline 
for $f\in \Tr(\Lambda)$ we have $G^\infty_m(f)\in \Tr(\Lambda)$ and 
$$
\|f-G^\infty_m(f)\|_\infty \le C_3(d)(\mb)^{-1/2}n^{1/2}|\Lambda|^{1/2}\|f\|_2,\quad \bar m := \max(m,1).
$$
\end{Theorem}

We now consider a dictionary 
$$
\D^3:=\D^3(Q):= \D^2(Q)\cup \D^\Tr(Q).
$$

\begin{Lemma}\label{L2.1} For any $Q\subset \Pi(\bN)$ with $\bN=(2^n,\dots,2^n)$ we have
$$
\sigma_m(\Tr(Q)_1,\D^3(Q))_\infty \le C(d)n^{1/2}|Q|/m
$$
and $|\D^3(Q)|\le C'(d)2^{nd}$.
\end{Lemma}
\begin{proof} Take $f\in\Tr(Q)_1$. Applying first Theorem \ref{T2.4} with $[m/2]$ and, then, applying 
Theorem \ref{T2.5} with $\Lambda = Q$ and $[m/2]$ we obtain
$$
\sigma_m(f,\D^3(Q))_\infty \ll n^{1/2}(| Q_n|/m)\|f\|_1,
$$
which proves the lemma.
\end{proof}

Lemma \ref{L2.1}, Theorem \ref{T2.1}, and Remark \ref{R2.1} imply the following 
result on the entropy numbers.

\begin{Theorem}\label{T2.6} For any $Q\subset \Pi(\bN)$ with $\bN=(2^n,\dots,2^n)$ we have
$$
\e_k(\Tr( Q)_1,L_\infty) \ll  \left\{\begin{array}{ll} n^{3/2}(|Q|/k), &\quad k\le 2| Q|,\\
n^{3/2}2^{-k/(2| Q|)},&\quad k\ge 2| Q|.\end{array} \right.
$$
\end{Theorem}

The above theorem with $Q=Q_n$ can be used for proving the upper bounds for the entropy numbers of the mixed smoothness classes. We define the classes which were studied in \cite{VT152} and \cite{VT156}.

 Let $\bs=(s_1,\dots,s_d)$ be a vector with nonnegative integer coordinates ($\bs \in \Z^d_+$) and as above
$$
\rho(\bs):= \{\bk=(k_1,\dots,k_d)\in \Z^d_+ : [2^{s_j-1}]\le |k_j|<2^{s_j},\quad j=1,\dots,d\}
$$
where $[a]$ denotes the integer part of a number $a$.  Define for $f\in L_1$
$$
\delta_\bs(f) := \sum_{\bk\in\rho(\bs)} \hat f(\bk) e^{i(\bk,\bx)},
$$
and
$$
f_l:=\sum_{\|\bs\|_1=l}\delta_\bs(f), \quad l\in \N_0,\quad \N_0:=\N\cup \{0\}.
$$
  Consider the class (see \cite{VT152})
$$
\bW^{a,b}_q:=\{f: \|f_l\|_q \le 2^{-al}(\bar l)^{(d-1)b}\},\quad \bar l:=\max(l,1).
$$
Define
$$
\|f\|_{\bW^{a,b}_q} := \sup_l \|f_l\|_q 2^{al}(\bar l)^{-(d-1)b}.
$$
Here is one more class, which is equivalent to $\bW^{a,b}_q$ in the case $1<q<\infty$ (see \cite{VT152}). 
Consider a class ${\bar \bW}^{a,b}_q$, which consists of functions $f$ with a representation  
$$
f=\sum_{n=1}^\infty t_n, \quad t_n\in \Tr(Q_n), \quad \|t_n\|_q \le 2^{-an} n^{b(d-1)}.
$$
  In the case $q=1$ classes ${\bar \bW}^{a,b}_1$ are wider than $\bW^{a,b}_1$. 
   
The following theorem was proved in \cite{VT156}.

\begin{Theorem}\label{T2.7} Let $d=2$ and $a>1$. Then
\be\label{2.7}
\e_k(\bW^{a,b}_1,L_\infty) \asymp \e_k(\bar\bW^{a,b}_1,L_\infty) \asymp k^{-a} (\log k)^{a+b+1/2}.
\ee
\end{Theorem}

We prove here an extension of Theorem \ref{T2.7} to all $d$. We note that this extension -- Theorem \ref{T2.8} -- is weaker than Theorem \ref{T2.7} in case $d=2$. 

\begin{Theorem}\label{T2.8} Let  $a>1$. Then
\be\label{2.8}
\e_k(\bW^{a,b}_1,L_\infty) \le \e_k(\bar\bW^{a,b}_1,L_\infty) \ll k^{-a} (\log k)^{(a+b)(d-1)+3/2}.
\ee
\end{Theorem}
\begin{proof} The proof is based on the following general result from \cite{VT156}. Let $X$ and $Y$ be two Banach spaces. We discuss a problem of estimating the entropy numbers of an approximation class, defined in the space $X$, in the norm of the space $Y$. Suppose a sequence of finite dimensional subspaces $X_n \subset X$, $n=1,\dots $, is given. Define the following class 
$$
{\bar \bW}^{a,b}_X:={\bar \bW}^{a,b}_X\{X_n\} := \{f\in X: f=\sum_{n=1}^\infty f_n,\quad  f_n\in X_n, 
$$
$$
 \|f_n\|_X \le 2^{-an}n^{b},\quad n=1,2,\dots\}.
$$
In particular,
$$
 {\bar \bW}^{a,b}_q = {\bar \bW}^{a,b(d-1)}_{L_q}\{\Tr(Q_n)\} .
$$ 
Denote $D_n:=\dim X_n$ and assume that for the unit balls $B(X_n):=\{f\in X_n: \|f\|_X\le 1\}$ we have the following upper bounds for the entropy numbers: there exist real $\al$ and nonnegative   $\ga$ and $\bt\in(0,1]$ such that 
\be\label{EA}
\e_k(B(X_n),Y) \ll n^\al \left\{\begin{array}{ll} (D_n/(k+1))^\bt (\log (4D_n/(k+1)))^\ga, &\quad k\le 2D_n,\\
2^{-k/(2D_n)},&\quad k\ge 2D_n.\end{array} \right.
\ee
\begin{Theorem}\label{T2.1G} Assume $D_n \asymp 2^n n^c$, $c\ge 0$, $a>\bt$, and subspaces $\{X_n\}$ satisfy (\ref{EA}). Then
\be\label{2.0G}
\e_k(\bar \bW^{a,b}_X\{X_n\},Y) \ll k^{-a} (\log k)^{ac+b+\al}.
\ee
\end{Theorem}

Theorem \ref{T2.6} with $Q=Q_n$ provides (\ref{EA}) with $\alpha=3/2$, $\beta=1$, $\gamma =0$. It remains to apply Theorem \ref{T2.1G} with $X_n=\Tr(Q_n)$ and $c=d-1$.
\end{proof}

\section{The Marcinkiewicz-type discretization theorem for the trigonometric polynomials}
\label{M}

 In this section we improve Theorem \ref{T3.2} from the Introduction, which was proved in \cite{VT160}, in two directions. We prove the Marcinkiewicz-type discretization theorem for $\Tr(Q)$ instead of $\Tr(Q_n)$ for a rather general $Q$. Also, even in a more general situation, we improve the bound from $m \le C_1(d)|Q_n|n^{d/2+3}$ to $m \le C_1(d)|Q_n|n^{7/2}$ similar to that in Theorem \ref{T3.1}. Our prove goes along the lines of the proof of Theorem \ref{T2.1} from \cite{VT160}. We use the following results from \cite{VT160}. Lemma \ref{L3.1} is from \cite{BLM}.

\begin{Lemma}\label{L3.1} Let $\{g_j\}_{j=1}^m$ be independent random variables with $\bE g_j=0$, $j=1,\dots,m$, which satisfy
$$
\|g_j\|_1\le 2,\qquad \|g_j\|_\infty \le M,\qquad j=1,\dots,m.
$$
Then for any $\eta \in (0,1)$ we have the following bound on the probability
$$
\bP\left\{\left|\sum_{j=1}^m g_j\right|\ge m\eta\right\} < 2\exp\left(-\frac{m\eta^2}{8M}\right).
$$
\end{Lemma}

We now consider measurable functions $f(\bx)$, $\bx\in \Omega$. For $1\le q<\infty$ define
$$
L^q_\bz(f) := \frac{1}{m}\sum_{j=1}^m |f(\bx^j)|^q -\|f\|_q^q,\qquad \bz:= (\bx^1,\dots,\bx^m).
$$
Let $\mu$ be a probabilistic measure on $\Omega$. Denote $\mu^m := \mu\times\cdots\times\mu$ the probabilistic measure on $\Omega^m := \Omega\times\cdots\times\Omega$.
We will need the following inequality, which is a corollary of Lemma \ref{L3.1} (see \cite{VT160}). 

\begin{Proposition}\label{P3.1} Let $f_j\in L_1(\Omega)$ be such that
$$
\|f_j\|_1 \le 1/2,\quad j=1,2;\qquad \|f_1-f_2\|_\infty \le \delta.
$$
Then
\be\label{3.1}
\mu^m\{\bz: |L^1_\bz(f_1) -L^1_\bz(f_2)| \ge \eta\} < 2\exp\left(-\frac{m\eta^2}{16\delta}\right).
\ee
\end{Proposition}

We now prove the Marcinkiewicz-type theorem for discretization of the $L_1$ norm of polynomials from $\Tr(Q)$. 
\begin{Theorem}\label{T3.3} For any $Q\subset \Pi(\bN)$ with $\bN=(2^n,\dots,2^n)$
 there exists a set of $m \le C_1(d)|Q|n^{7/2}$ points $\xi^j\in \T^d$, $j=1,\dots,m$ such that for any $f\in \Tr(Q)$ 
we have
$$
C_2\|f\|_1 \le \frac{1}{m}\sum_{j=1}^m |f(\xi^j)| \le C_3\|f\|_1.
$$
\end{Theorem}
\begin{proof} We use the technique developed in learning theory and in distribution-free theory of regression known under the name of {\it chaining technique}. Proposition \ref{P3.1} plays an important role in our proof. It is used in the proof of the bound on the probability of the event 
$\{\sup_{f\in W}|L^1_\bz(f)|\ge \eta\}$ for a function class $W$. The corresponding proof is
  in terms of the entropy numbers of $W$.    
 
 We consider the case $X$ is 
 $\C(\Omega)$ the space of functions continuous on a compact subset $\Omega$ of $\bR^d$ with the norm
$$
\|f\|_\infty:= \sup_{\bx\in \Omega}|f(\bx)|.
$$
  We use the abbreviated notations
$$
  \e_n(W):= \e_n(W,\C).
$$
In our case  
\be\label{3.3}
W:=W(Q) := \{t\in \Tr(Q): \|t\|_1 = 1/2\}.
\ee
We use Theorem \ref{T2.6} proved in Section \ref{entropy}. We formulate it here for the reader's convenience. We stress that Theorem \ref{T2.6'} is the only result on the specific features of the $\Tr(Q)$, which we use in the proof of Theorem \ref{T3.3}. 
\begin{Theorem}\label{T2.6'} For any $Q\subset \Pi(\bN)$ with $\bN=(2^n,\dots,2^n)$ we have
$$
\e_k(\Tr( Q)_1,L_\infty) \le 2\e_k:= 2C_4(d)  \left\{\begin{array}{ll} n^{3/2}(|Q|/k), &\quad k\le 2| Q|,\\
n^{3/2}2^{-k/(2| Q|)},&\quad k\ge 2| Q|.\end{array} \right.
$$
\end{Theorem}
Specify $\eta=1/4$.
Denote $\de_j := \e_{2^j}$, $j=0,1,\dots$, and consider minimal $\de_j$-nets ${\mathcal N}_j \subset W$ of $W$ in $\C(\T^d)$. We use the notation $N_j:= |{\mathcal N}_j|$. Let $J$ be the minimal $j$ satisfying $\de_j \le 1/16$. For $j=1,\dots,J$ we define a mapping $A_j$ that associates with a function $f\in W$ a function $A_j(f) \in {\mathcal N}_j$ closest to $f$ in the $\C$ norm. Then, clearly,
$$
\|f-A_j(f)\|_\C \le \de_j.  
$$
We use the mappings $A_j$, $j=1,\dots, J$ to associate with a function $f\in W$ a sequence (a chain) of functions $f_J, f_{J-1},\dots, f_1$ in the following way
$$
f_J := A_J(f),\quad f_j:=A_j(f_{j+1}),\quad j=J-1,\dots,1.
$$
Let us find an upper bound for $J$, defined above. Certainly, we can carry out the proof under assumption that $C_4(d)\ge 1$. Then the definition of $J$ implies that $2^J\ge 2|Q|$ and
\be\label{3.5}
C_4(d)n^{3/2}2^{-2^{J-1}/(2| Q|)} \ge 1/16.
\ee
We derive from (\ref{3.5})
\be\label{3.6}
2^J \le 4|Q| C(d)\log n,\qquad J \le 2dn 
\ee
for sufficiently large $n\ge C(d)$. 

Set 
$$
\eta_j := \frac{1}{16nd},\quad j=1,\dots,J.
$$

We now proceed to the estimate of $\mu^m\{\bz:\sup_{f\in W}|L^1_\bz(f)|\ge 1/4\}$. First of all   by the following simple  Proposition \ref{P3.2}   the assumption $\de_J\le 1/16$ implies that if $|L^1_\bz(f)| \ge 1/4$ then $|L^1_\bz(f_J)|\ge 1/8$. 
\begin{Proposition}\label{P3.2} If $\|f_1-f_2\|_\infty\le \delta$, then
$$
|L^1_\bz(f_1)-L^1_\bz(f_2)| \le 2\delta.
$$
\end{Proposition}
Rewriting 
$$
L^1_\bz(f_J) = L^1_\bz(f_J)-L^1_\bz(f_{J-1}) +\dots+L^1_\bz(f_{2})-L^1_\bz(f_1)+L^1_\bz(f_1)
$$
we conclude that if $|L^1_\bz(f)| \ge 1/4$ then at least one of the following events occurs:
$$
|L^1_\bz(f_j)-L^1_\bz(f_{j-1})|\ge \eta_j\quad\text{for some}\quad j\in (1,J] \quad\text{or}\quad |L^1_\bz(f_1)|\ge \eta_1.
$$
Therefore
\begin{eqnarray}\label{3.6'}
\mu^m\{\bz:\sup_{f\in W}|L^1_\bz(f)|\ge1/4\}
\le \mu^m\{\bz:\sup_{f\in {\mathcal N}_1}|L^1_\bz(f)|\ge\eta_1\} \nonumber \\
+\sum_{j\in(1,J]}\sum_{f\in {\mathcal N}_j}\mu^m
\{\bz:|L^1_\bz(f)-L^1_\bz(A_{j-1}(f))|\ge\eta_j\}\nonumber\\
\le \mu^m\{\bz:\sup_{f\in {\mathcal N}_1}|L^1_\bz(f)|\ge\eta_1\}\nonumber\\
+\sum_{j\in(1,J]} N_j\sup_{f\in W}\mu^m
\{\bz:|L^1_\bz(f)-L^1_\bz(A_{j-1}(f))|\ge\eta_j\}. 
\end{eqnarray}
  Applying  Proposition \ref{P3.1} we obtain
$$
\sup_{f\in W} \mu^m\{\bz:|L^1_\bz(f)-L^1_\bz(A_{j-1}(f))|\ge \eta_j\} \le 2\exp\left(-\frac{m\eta_j^2}{16\de_{j-1}}\right).
$$

We now make further estimates for a specific $m=C_1(d)|Q|n^{7/2}$ with large 
enough $C_1(d)$. For $j$ such that $2^j\le 2|Q|$ we obtain from the definition of 
$\delta_j$
$$
\frac{m\eta_j^2}{\delta_{j-1}} \ge \frac{C_1(d)n^{3/2}2^{j-1}}{C_5(d)n^{3/2} } \ge \frac{C_1(d)}{2C_5(d)}2^j.
$$
By our choice of $\delta_j=\e_{2^j}$ we get $N_j\le 2^{2^j} <e^{2^j}$ and, therefore,
\be\label{3.7}
N_j\exp\left(-\frac{m\eta_j^2}{16\de_{j-1}}\right)\le \exp(-2^j)
\ee
for sufficiently large $C_1(d)$.

In the case $2^j\in (2|Q|, 2^J]$ we have
$$
\frac{m\eta_j^2}{\delta_{j-1}} \ge \frac{C_1(d)|Q|n^{3/2}}{C_6(d)n^{3/2}2^{-2^{j-1}/(2|Q|)}} \ge \frac{C_1(d)}{C_7(d)}2^j
$$
   and
\be\label{3.8}
N_j\exp\left(-\frac{m\eta_j^2}{16\de_{j-1}}\right)\le \exp(-2^j)
\ee
for sufficiently large $C_1(d)$.

We now estimate $\mu^m\{\bz:\sup_{f\in {\mathcal N}_1}|L^1_\bz(f)|\ge\eta_1\}$.
We use Lemma \ref{L3.1} with $g_j(\bz) = |f(\bx^j)|-\|f\|_1$. To estimate $\|g_j\|_\infty$ it is sufficient to use the following trivial Nikol'skii-type inequality for the 
trigonometric polynomials: 
\be\label{3.9}
\|f\|_\infty \le |Q| \|f\|_1,\qquad f\in \Tr(Q).
\ee
  Then Lemma \ref{L3.1} gives
$$
\mu^m\{\bz:\sup_{f\in {\mathcal N}_1}|L^1_\bz(f)|\ge\eta_1\}\le N_1\exp\left(-\frac{m\eta_1^2}{C|Q|}\right) \le 1/4
$$
for sufficiently large $C_1(d)$. Substituting the above estimates into (\ref{3.6'}) we obtain
$$
\mu^m\{\bz:\sup_{f\in W}|L^1_\bz(f)|\ge1/4\} <1.
$$
Therefore, there exists $\bz_0=(\xi^1,\dots,\xi^m)$ such that for any $f\in W$ we have 
$$
|L^1_{\bz_0}(f)| \le 1/4.
$$
Taking into account that $\|f\|_1=1/2$ for $f\in W$ we obtain the statement of Theorem \ref{T3.3} with $C_2 =1/2$, $C_3=3/2$.
 
\end{proof}

In the above proof of Theorem \ref{T3.3} we specified $\eta =1/4$. If instead we 
take $\eta \in [2^{-2^{nd/2}},1/4]$, define $J(\eta)$ to be the minimal $j$ satisfying $\delta_j \le \eta/4$ and set 
$$
\eta_j := \frac{\eta}{4nd},
$$
then we obtain the following generalization of Theorem \ref{T3.3}.

\begin{Theorem}\label{T3.3e} For any $Q\subset \Pi(\bN)$ with $\bN=(2^n,\dots,2^n)$ and $\epsilon \in [2^{1-2^{nd/2}},1/2]$
 there exists a set of $m \le C_1(d)|Q|n^{7/2}\epsilon^{-2}$ points $\xi^j\in \T^d$, $j=1,\dots,m$ such that for any $f\in \Tr(Q)$ 
we have
$$
(1-\epsilon)\|f\|_1 \le \frac{1}{m}\sum_{j=1}^m |f(\xi^j)| \le (1+\epsilon)\|f\|_1.
$$
\end{Theorem}

\section{Some Marcinkiewicz-type discretization theorems for general polynomials}
\label{MX}

In this section we extend the technique developed in Sections \ref{entropy} and \ref{M} to the case of a general orthonormal system $\{u_i\}_{i=1}^N$ on a compact $\Omega \subset \R^d$, which satisfies  conditions {\bf A}, {\bf B}, and {\bf C} from the Introduction. Let $\mu$ be a probability measure on $\Omega$. It is convenient for us to assume that $u_i$, $i=1,\dots,N$, are real functions and denote
$$
\<u,v\> := \int_\Omega u(\bx)v(\bx) d\mu, \quad \|u\|_2 := \<u,u\>^{1/2}.
$$


 

Denote the unit $L_p$ ball in $X_N$ by 
$$
X^p_N :=\{f\in X_N: \|f\|_p\le 1\}.
$$
We begin with the estimates of the entropy numbers $\e_k(X^1_N,L_\infty)$. 
We use the same strategy as above: first we get bounds on $m$-term approximations for $X^1_N$ in $L_2$ with respect to a dictionary $\D^1$, second we obtain 
bounds on $m$-term approximations for $X^2_N$ in $L_\infty$ with respect to a dictionary $\D^2$, third we get  bounds on $m$-term approximations for $X^1_N$ in $L_\infty$ with respect to a dictionary $\D^3=\D^1\cup \D^2$. Then we apply 
Theorem \ref{2.1} to obtain the entropy numbers estimates. 

\subsection{Sparse approximation in $L_2$} 
We begin with the study of $m$-term approximations with respect to the dictionary
$$
\D^0:= \{g_\by(\bx)\}_{\by\in\Omega},\quad g_\by(\bx):= (K_2N)^{-1/2}\D_N(\cdot,\by),
$$
where
$$
\D_N(\bx,\by) := \sum_{i=1}^N u_i(\bx)u_i(\by)
$$
is the Dirichlet kernel for the system $\{u_i\}_{i=1}^N$. Then assumption {\bf B} guarantees that $\|g_\by\|_2 \le 1$. We now use the following greedy-type algorithm (see \cite{Tbook}, p.82).

{\bf Relaxed Greedy Algorithm (RGA).}  Let $f^r_0:=f$
and $G_0^r(f):= 0$. For a  function $h$ from a real Hilbert space $H$, let $g=g(h)$ denote the function from $\D^\pm:=\{\pm g:g\in\D\}$, which
maximizes $\langle h,g\rangle$ (we assume the existence of such an element).  Then, for each $m\ge 1$, we inductively define 
$$ 
 G_m^r(f):= 
\left(1-\frac{1}{m}\right)G_{m-1}^r(f)+\frac{1}{m}g(f_{m-1}^r), \quad
f_m^r:= f-G_m^r(f). 
$$
We use the following known result (see \cite{Tbook}, p.90).
\begin{Theorem}\label{T4.1} For the Relaxed Greedy Algorithm we have, for each $f\in
A_1(\D^\pm)$, the estimate
 $$
\|f-G_m^r(f)\|\le \frac{2}{\sqrt{m}},\quad m\ge 1.
$$
\end{Theorem}

In our application of the above RGA the Hilbert space $H$ is the $X_N$ with the $L_2$ norm, the dictionary $\D$ is the $\D^0$ defined above. Using representation
\be\label{4.4}
f(\bx) = \int_{\Omega} f(\by) \D_N(\bx,\by)d\mu(\by)
\ee
we see that the search for $g\in(\D^0)^\pm$ maximizing $\<h,g\>$, $h\in X_N$, is equivalent to the search for $\by\in\Omega$ maximizing $|h(\by)|$. A function 
$h$ from $X_N$ is continuos on the compact $\Omega$ and, therefore, such 
a maximizing $\by_{\max}$ exists. This means that we can run the RGA. 

For $f\in X^1_N$ by representation (\ref{4.4})  we obtain
$$
f(\bx) = \int_{\Omega} f(\by) \D_N(\bx,\by)d\mu(\by)  
$$
$$
= (K_2N)^{1/2}\int_{\Omega} f(\by) (K_2N)^{-1/2}\D_N(\bx,\by)d\mu(\by).
$$
Therefore,
$$
(K_2N)^{-1/2} f \in A_1((\D^0)^\pm),\quad \text{or}\quad f\in A_1((\D^0)^\pm,(K_2N)^{1/2}).
$$
Applying Theorem \ref{T4.1} we get the following result.
\begin{Theorem}\label{T4.2} For the Relaxed Greedy Algorithm with respect to $\D^0$ we have, for each $f\in X^1_N$, the estimate
 $$
\|f-G_m^r(f)\|\le 2(K_2N/m)^{1/2},\quad m\ge 1.
$$
\end{Theorem}
 We need an analog of Theorem \ref{T4.2} for a discrete version of $\D^0$. 
 Take a $\delta>0$ and let $\{\by^1,\dots,\by^M\}$, $M=M(\delta)$, be a $\delta$-net of points in $\Omega$, which means that for any $\by\in \Omega$ there is a $\by^j$ from the net such that $\|\by-\by^j\|_\infty\le \delta$. It is clear that 
 \be\label{4.5}
 M(\delta) \le (C(\Omega)/\delta)^d.
 \ee
  It follows from the definition of the RGA that $G_m^r(f)\in A_1(\D^\pm)$ provided 
  $f\in A_1(\D^\pm)$. Let $f\in X^1_N$ and let $G_m^r(f)$ be its approximant from Theorem \ref{T4.2}. Then
\be\label{4.6}
  G_m^r(f) = \sum_{k=1}^m c_kg_{\by(k)},\quad \sum_{k=1}^m |c_k| \le (K_2N)^{1/2}.
\ee
  For each $\by(k)$ find $\by^{j(k)}$ from the net such that $\|\by(k)-\by^{j(k)}\|_\infty \le \delta$. Then, using assumption {\bf A} we get
$$ 
\|g_{\by(k)}-g_{\by^{j(k)}}\|_2^2 = (K_2N)^{-1}\sum_{i=1}^N |u_i(\by(k))-u_i(\by^{j(k)})|^2 
$$
\be\label{4.7} 
\le (K_2N)^{-1} K_1^2 N^{1+2\beta} \delta^{2\alpha}.
\ee
Denote
$$
t_m(f) := \sum_{k=1}^m c_kg_{\by^{j(k)}}.
$$
Combining (\ref{4.6}) and (\ref{4.7}) we obtain
\be\label{4.8}
\|G_m^r(f)-t_m(f)\|_2 \le (K_2N)^{1/2}K_2^{-1/2} K_1N^\beta\delta^\alpha.
\ee
Choosing $\delta$ such that 
$$
\delta_0^\alpha = K_1^{-1}N^{-1/2-\beta}  
$$
we obtain by Theorem \ref{T4.2} and (\ref{4.8}) that for $f\in X^1_N$
\be\label{4.9}
\|f-t_m(f)\|_2 \le 3(K_2N/m)^{1/2},\quad m\le N.
\ee
Inequality (\ref{4.5}) gives
\be\label{4.10}
M(\delta_0) \le C(K_1,\Omega,d) N^{c(\alpha,\beta,d)}.
\ee
Define the dictionary $\D^1$ as follows
$$
\D^1 := \{g_{\by^j}\}_{j=1}^M.
$$
Relation (\ref{4.9}) gives us the following theorem.
\begin{Theorem}\label{T4.3} We have
$$
\sigma_m(X^1_N,\D^1)_2 \le 3(K_2N/m)^{1/2}.
$$
\end{Theorem}

\subsection{Sparse approximation in $L_\infty$}

In this subsection we study $m$-term approximations of $f\in X^2_N$ in the $L_\infty$ norm with respect to the following dictionary
$$
\D^2 := \{\pm g_i\}_{i=1}^N,\quad g_i:= u_i K_2^{-1/2}.
$$
Then by property {\bf B} for all $p$ we have $\|g_i\|_p \le 1$. 

In this subsection we use greedy algorithms in Banach spaces. We remind some notations from the theory of greedy approximation in Banach spaces. The reader can find a systematic presentation of this theory in \cite{Tbook}, Chapter 6. 
Let $X$ be a Banach space with norm $\|\cdot\|$. We say that a set of elements (functions) ${\mathcal D}$ from $X$ is a dictionary   if each $g\in {\mathcal D}$ has norm less than or equal to one ($\|g\|\le 1$)
 and the closure of $\sp {\mathcal D}$ coincides with $X$. We note that in \cite{T8} we required in the definition of a dictionary normalization of its elements ($\|g\|=1$). However, it is pointed out in \cite{T12} that it is easy to check that the arguments from \cite{T8} work under assumption $\|g\|\le 1$ instead of $\|g\|=1$. In applications  it is more convenient for us to have an assumption $\|g\|\le 1$ than normalization of a dictionary.

For an element $f\in X$ we denote by $F_f$ a norming (peak) functional for $f$: 
$$
\|F_f\| =1,\qquad F_f(f) =\|f\|.
$$
The existence of such a functional is guaranteed by the Hahn-Banach theorem.

We proceed to  the Incremental Greedy Algorithm (see \cite{T12} and \cite{Tbook}, Chapter 6).       Let $\ep=\{\ep_n\}_{n=1}^\infty $, $\ep_n> 0$, $n=1,2,\dots$ . For a Banach space $X$ and a dictionary $\Di$ define the following algorithm IA($\ep$) $:=$ IA($\ep,X,\Di$).

 {\bf Incremental Algorithm with schedule $\ep$ (IA($\ep,X,\Di$)).} 
  Denote $f_0^{i,\ep}:= f$ and $G_0^{i,\ep} :=0$. Then, for each $m\ge 1$ we have the following inductive definition.

(1) $\ff_m^{i,\ep} \in \Di$ is any element satisfying
$$
F_{f_{m-1}^{i,\ep}}(\ff_m^{i,\ep}-f) \ge -\ep_m.
$$

(2) Define
$$
G_m^{i,\ep}:= (1-1/m)G_{m-1}^{i,\ep} +\ff_m^{i,\ep}/m.
$$

(3) Let
$$
f_m^{i,\ep} := f- G_m^{i,\ep}.
$$
We consider here approximation in uniformly smooth Banach spaces. For a Banach space $X$ we define the modulus of smoothness
$$
\rho(u) := \sup_{\|x\|=\|y\|=1}(\frac{1}{2}(\|x+uy\|+\|x-uy\|)-1).
$$
The uniformly smooth Banach space is the one with the property
$$
\lim_{u\to 0}\rho(u)/u =0.
$$

It is well known (see for instance \cite{DGDS}, Lemma B.1) that in the case $X=L_p$, 
$1\le p < \infty$ we have
\begin{equation}\label{t6.2}
\rho(u) \le \begin{cases} u^p/p & \text{if}\quad 1\le p\le 2 ,\\
(p-1)u^2/2 & \text{if}\quad 2\le p<\infty. \end{cases}     
\end{equation}
 
 Denote by $A_1({\mathcal D}):=A_1(\Di)(X)$ the closure in $X$ of the convex hull of ${\mathcal D}$.
 In order to be able to run the IA($\ep$) for all iterations we need existence of an element $\ff_m^{i,\ep} \in \Di$ at the step (1) of the algorithm for all $m$. It is clear that the following condition guarantees such existence (see \cite{VT149}).

{\bf Condition B.} We say that for a given dictionary $\Di$ an element $f$ satisfies Condition B if for all $F\in X^*$ we have
$$
F(f) \le \sup_{g\in\Di}F(g).
$$

It is well known (see, for instance, \cite{Tbook}, p.343) that any $f\in A_1(\Di)$ satisfies Condition B. For completeness we give this simple argument here. 
Take any $f\in A_1(\Di)$. Then for any $\ep >0$ there exist $g_1^\ep,\dots,g_N^\ep \in \Di$ and numbers $a_1^\ep,\dots,a_N^\ep$ such that $a_i^\ep>0$, $a_1^\ep+\dots+a_N^\ep = 1$ and 
$$
\|f-\sum_{i=1}^Na_i^\ep g_i^\ep\| \le \ep.
$$
Thus
$$
F(f) \le \|F\|\ep + F(\sum_{i=1}^Na_i^\ep g_i^\ep) \le \ep \|F\| +\sup_{g\in \Di} F(g)
$$
which proves Condition B.
 
 We note that Condition B is equivalent to the property $f\in A_1(\Di)$. Indeed, as we showed above, the property $f\in A_1(\Di)$ implies Condition B. Let us show that Condition B implies that $f\in A_1(\Di)$. Assuming the contrary $f\notin A_1(\Di)$ by the separation theorem for convex bodies we find $F\in X^*$ such that 
 $$
 F(f) > \sup_{\phi\in A_1(\Di)} F(\phi) \ge \sup_{g\in\Di}F(g)
 $$
 which contradicts Condition B.
 
 We formulate results on the IA($\ep$) in terms of Condition B because in the applications it is easy to check Condition B.
 
\begin{Theorem}\label{T4.4} Let $X$ be a uniformly smooth Banach space with  modulus of smoothness $\rho(u)\le \gamma u^q$, $1<q\le 2$. Define
$$
\ep_n := \bt\gamma ^{1/q}n^{-1/p},\qquad p=\frac{q}{q-1},\quad n=1,2,\dots .
$$
Then, for every $f$ satisfying Condition B we have
$$
\|f_m^{i,\ep}\| \le C(\bt) \gamma^{1/q}m^{-1/p},\qquad m=1,2\dots.
$$
\end{Theorem}

In the case $f\in A_1(\Di)$ this theorem is proved in \cite{T12} (see also \cite{Tbook}, Chapter 6). As we mentioned above Condition B is equivalent to $f\in A_1(\Di)$. 

For $f\in X_N$ write $f=\sum_{i=1}^N c_ig_i$ and define
$$
\|f\|_A := \sum_{i=1}^N |c_i|.
$$

 \begin{Theorem}\label{T4.5} Assume that $X_N$ satisfies {\bf C}. For any $t\in X_N$ the IA($\ep,X_N\cap L_p,\D^2$) with an appropriate $p$ and schedule $\ep$,  applied to $f:=t/\|t\|_A$, provides after $m$ iterations an $m$-term polynomial $G_m(t):=G^{i,\ep}_m(f)\|t\|_A$ with the following approximation property
$$
\|t-G_m(t)\|_\infty \le Cm^{-1/2}(\ln N)^{1/2}\|t\|_A, \quad \|G_m(t)\|_A =\|t\|_A, 
$$
with a constant $C=C(K_3,K_4)$.
\end{Theorem}
\begin{proof}  It is clear that it is sufficient to prove Theorem \ref{T4.5} for $t\in X_N$ with $\|t\|_A =1$. Then $t\in A_1(\D^2)(X_N\cap L_p)$ for all $p\in [2,\infty)$.  
Applying the IA($\ep$) to $f$ with respect to $\Di^2$ we obtain by Theorem \ref{T4.4} after $m$ iterations 
 \begin{equation}\label{4.12}
\| t -\sum_{j\in \La} \frac{a_j}{m}g_j\|_p \le C\gamma^{1/2}m^{-1/2},\qquad \sum_{j\in\La}a_j = m,
\end{equation}
where $\sum_{j\in \La} \frac{a_j}{m}g_j$ is the $G^{i,\ep}_m(t)$. 
By (\ref{t6.2}) we find $\gamma \le p/2$. Next, by the Nikol'skii inequality from assumption {\bf C} we get from (\ref{4.12})
$$
\| t -\sum_{j\in \La} \frac{a_j}{m}g_j\|_\infty \le CN^{K_4/p}\| t -\sum_{j\in \La} \frac{a_j}{m}g_j\|_p \le Cp^{1/2}N^{K_4/p}m^{-1/2}.
$$
Choosing $p\asymp \ln N$ we obtain the desired in Theorem \ref{T4.5} bound.

\end{proof}

Using the following simple relations
$$
\|f\|_2^2 = \|\sum_{i=1}^N c_ig_i\|_2^2 = \|K_2^{-1/2}\sum_{i=1}^N c_iu_i\|_2^2 = 
K_2^{-1}\sum_{i=1}^N |c_i|^2,
$$
$$
\sum_{i=1}^N |c_i| \le N^{1/2} \left(\sum_{i=1}^N |c_i|^2\right)^{1/2} = (K_2N)^{1/2}\|f\|_2
$$
we obtain from Theorem \ref{T4.5} the following estimates.

\begin{Theorem}\label{T4.6} We have
$$
\sigma_m(X^2_N,\D^2)_\infty \ll (N/m)^{1/2} (\ln N)^{1/2}.
$$
\end{Theorem}

Combining Theorems \ref{T4.3} and \ref{T4.6} we obtain.

\begin{Theorem}\label{T4.7} We have
$$
\sigma_m(X^1_N,\D^1\cup\D^2)_\infty \ll (N/m) (\ln N)^{1/2}.
$$
\end{Theorem}

\subsection{The entropy numbers}

By our construction (see (\ref{4.10})) we obtain
$$
|\D^1\cup\D^2| \ll N^{c(\alpha,\beta,d)}.
$$

Theorem \ref{T4.7}, Theorem \ref{T2.1}, and Remark \ref{R2.1} (its version for the real case) imply the following 
result on the entropy numbers.

\begin{Theorem}\label{T4.8} Suppose that a real orthonormal system $\{u_i\}_{i=1}^N$ satisfies conditions {\bf A}, {\bf B}, and {\bf C}. 
 Then we have
$$
\e_k(X^1_N,L_\infty) \ll  \left\{\begin{array}{ll} (\log N)^{3/2}(N/k), &\quad k\le N,\\
(\log N)^{3/2}2^{-k/N},&\quad k\ge N.\end{array} \right.
$$
\end{Theorem}

In the same way as Theorem \ref{T3.3} was derived from Theorem \ref{T2.6} the following Theorem \ref{T4.9} can be derived from Theorem \ref{T4.8}

\begin{Theorem}\label{T4.9} Suppose that a real orthonormal system $\{u_i\}_{i=1}^N$ satisfies conditions {\bf A}, {\bf B}, and {\bf C}. Then there exists a set of $m \le C_1N(\log N)^{7/2}$ points $\xi^j\in \Omega$, $j=1,\dots,m$, $C_1=C(d,K_1,K_2,K_3,K_4,\Omega,\alpha,\beta)$, such that for any $f\in X_N$ 
we have
$$
C_2\|f\|_1 \le \frac{1}{m}\sum_{j=1}^m |f(\xi^j)| \le C_3\|f\|_1.
$$
\end{Theorem}

The following analog of Theorem \ref{T3.3e} holds for general systems.

\begin{Theorem}\label{T4.9e} Suppose that a real orthonormal system $\{u_i\}_{i=1}^N$ satisfies conditions {\bf A}, {\bf B}, and {\bf C}. Then for $\epsilon \in [2^{-N},1/2]$ there exists a set of \newline $m \le C_1N(\log N)^{7/2}\epsilon^{-2}$ points $\xi^j\in \Omega$, $j=1,\dots,m$, \newline$C_1=C(d,K_1,K_2,K_3,K_4,\Omega,\alpha,\beta)$, such that for any $f\in X_N$ 
we have
$$
(1-\epsilon)\|f\|_1 \le \frac{1}{m}\sum_{j=1}^m |f(\xi^j)| \le (1+\epsilon)\|f\|_1.
$$
\end{Theorem}

\subsection{Conditional theorem}  

We already pointed out in the proof of Theorem \ref{T3.3} that the only special 
properties of the subspace $\Tr(Q)$, which we used in the proof of Theorem \ref{T3.3}, were stated in Theorem \ref{T2.6} on the entropy numbers $\e_k(\Tr(Q)_1,L_\infty)$. Similarly, in Section \ref{MX} above we used assumptions {\bf A}, {\bf B}, and {\bf C} to prove (constructively) Theorem \ref{T4.8} on the entropy numbers  $\e_k(X_N^1,L_\infty)$ and, then, derived from it Theorem \ref{T4.9}. This encourages us to formulate the following conditional result. 

\begin{Theorem}\label{T4.10} Suppose that a real $N$-dimensional subspace $X_N$ satisfies the condition $(B\ge 1)$
$$
\e_k(X^1_N,L_\infty) \le  B\left\{\begin{array}{ll}  N/k, &\quad k\le N,\\
 2^{-k/N},&\quad k\ge N.\end{array} \right.
$$
Then there exists a set of $m \le C_1NB(\log_2(2N\log_2(8B)))^2$ points $\xi^j\in \Omega$, $j=1,\dots,m$, with large enough absolute constant $C_1$, such that for any $f\in X_N$ 
we have
$$
\frac{1}{2}\|f\|_1 \le \frac{1}{m}\sum_{j=1}^m |f(\xi^j)| \le \frac{3}{2}\|f\|_1.
$$
\end{Theorem}

\section{The Marcinkiewicz-type theorem in $L_2$}
\label{L2}

In this section we discuss some known results directly connected with the discretization theorems and demonstrate how recent results on random matrices can be used to obtain the Marcinkiewicz-type theorem in $L_2$.
We begin with formulation of the Rudelson result from \cite{Rud}. In the paper \cite{Rud} it is formulated in terms of submatrices of an orthogonal matrix. We reformulate it in our notations. Let $\Omega_M=\{x^j\}_{j=1}^M$ be a discrete set with the probability measure $\mu(x^j)=1/M$, $j=1,\dots,M$. Assume that 
$\{u_i(x)\}_{i=1}^N$ is a real orthonormal on $\Omega_M$ system satisfying the following condition: for all $j$
\be\label{5.1}
\sum_{i=1}^Nu_i(x^j)^2 \le Nt^2
\ee
with some $t\ge 1$.
Then for every $\ep>0$ there exists a set $J\subset \{1,\dots,M\}$ of indices with cardinality 
\be\label{5.1a}
m:=|J| \le C\frac{t^2}{\ep^2}N\log\frac{Nt^2}{\ep^2}
\ee
such that for any $f=\sum_{i=1}^N c_iu_i$ we have
$$
(1-\ep)\|f\|_2^2 \le \frac{1}{m} \sum_{j\in J} f(x^j)^2 \le (1+\ep)\|f\|_2^2.
$$
In particular, the above result implies that for any orthonormal system $\{u_i\}_{i=1}^N$ on $\Omega_M$, satisfying (\ref{5.1}) we have
$$
{\mathcal U}_N := \sp(u_1,\dots,u_N) \in \cM(m,2)\quad \text{provided}\quad m\ge CN\log N
$$
with large enough $C$.
We note that (\ref{5.1}) is satisfied if the system $\{u_i\}_{i=1}^N$ is uniformly bounded: $\|u_i\|_\infty \le t$, $i=1,\dots,N$. 

We first demonstrate how the Bernstein-type concentration inequality for matrices can be used to prove an analog of the above Rudelson's result for a general $\Omega$. Our proof is based on a different idea than the Rudelson's proof. 
Let $\{u_i\}_{i=1}^N$ be an orthonormal system on $\Omega$, satisfying the condition

{\bf D.} For $x\in \Omega$ we have 
\be\label{5.2a}
w(x):=\sum_{i=1}^N u_i(x)^2 = N.
\ee
 With each $x\in\Omega$ we associate the matrix $G(x) := [u_i(x)u_j(x)]_{i,j=1}^N$. Clearly, $G(x)$ is a symmetric matrix. We will also need the matrix $G(x)^2$. We have for the $(k,l)$ element of $G(x)^2$
 $$
 (G(x)^2)_{k,l} = \sum_{j=1}^N u_k(x)u_j(x)u_j(x)u_l(x) = w(x)u_k(x)u_l(x).
 $$
 Therefore,
 \be\label{5.2}
 G(x)^2 = w(x)G(x)\quad \text{and}\quad \|G(x)\| = w(x).
 \ee
 We use the following Bernstein-type concentration inequality for matrices (see \cite{Tro12}).
 
 \begin{Theorem}\label{T5.1} Let $\{T_k\}_{k=1}^n$ be a sequence of independent random symmetric $N\times N$ matrices. Assume that each $T_k$ satisfies:
 $$
 \bE(T_k) =0 \quad \text{and}\quad \|T_k\| \le R \quad \text{almost surely}.
 $$
 Then for all $\eta\ge 0$
 $$
 \bP\left\{\left\|\sum_{k=1}^n T_k \right\|\ge \eta\right\} \le N\exp\left(-\frac{\eta^2}{2\sigma^2 +(2/3)R\eta}\right)
 $$
 where $\sigma^2 := \left\|\sum_{k=1}^n \bE(T_k^2)\right\|$.
 \end{Theorem}
 
 We now consider a sequence $T_k := G(x^k)-I$, $k=1,\dots,m$ of independent 
 random symmetric matrices. Orthonormality of the system $\{u_i\}_{i=1}^N$ implies that $\bE(T_k)=0$ for all $k$. Relation (\ref{5.2}) and our assumption {\bf D} imply for all $k$
 \be\label{5.3}
\|T_k\|\le \|G(x^k)\| +1 = N+1=: R.
\ee
 
Denote $T(x):= G(x)-I$ and, using (\ref{5.2a}) and (\ref{5.2}), represent
$$
T(x)^2 = G(x)^2 -2G(x) +I = (N-2)G(x) +I.
$$
Then by the orthonormality of the system $\{u_i\}_{i=1}^N$ we get
$$
\bE(T(x)^2) = (N-1)I
$$
 and, therefore, we obtain
\be\label{5.6}
\|\bE(T^2)\| \le N-1.
\ee
Thus, by Theorem \ref{T5.1} we obtain for $\eta \le 1$
\be\label{5.8}
\bP\left\{\left\|\sum_{k=1}^n (G(x^k)-I) \right\|\ge n\eta\right\} \le N\exp\left(-\frac{n\eta^2}{cN}\right)
\ee
with an absolute constant $c$.

For a set of points $\xi^k\in \Omega$, $k=1,\dots,m$, and $f=\sum_{i=1}^N b_iu_i$ we have
$$
\frac{1}{m}\sum_{k=1}^m f(\xi^k)^2 - \int_\Omega f(x)^2 d\mu = {\mathbf b}^T\left(\frac{1}{m}\sum_{k=1}^m G(\xi^k)-I\right){\mathbf b},
$$
where ${\mathbf b} = (b_1,\dots,b_N)^T$ is the column vector. Therefore,
\be\label{5.8'}
\left|\frac{1}{m}\sum_{k=1}^m f(\xi^k)^2 - \int_\Omega f(x)^2 d\mu \right| \le  
\left\|\frac{1}{m}\sum_{k=1}^m G(\xi^k)-I\right\|\|{\mathbf b}\|_2^2.
\ee

We now make $m = [CN\eta^{-2}\log N]$ with large enough $C$. Then, using 
(\ref{5.8}) with $n=m$, we get the corresponding probability $<1$. Thus, we have proved the following theorem.

\begin{Theorem}\label{T5.2} Let $\{u_i\}_{i=1}^N$ be an orthonormal system, satisfying condition {\bf D}. Then for every $\ep>0$ there exists a set $\{\xi^j\}_{j=1}^m \subset \Omega$ with
$$
m  \le C \ep^{-2}N\log N
$$
such that for any $f=\sum_{i=1}^N c_iu_i$ we have
$$
(1-\ep)\|f\|_2^2 \le \frac{1}{m} \sum_{j=1}^m f(\xi^j)^2 \le (1+\ep)\|f\|_2^2.
$$
\end{Theorem}

We note that Theorem \ref{T5.2} treats a special case, when (\ref{5.2a}) instead of (\ref{5.1}) is satisfied. This is the case, for instance, for the trigonometric and the Walsh systems. In this special case Theorem \ref{T5.2} is more general and slightly stronger than the Rudelson theorem discussed in the beginning of this section. Theorem \ref{T5.2} provides the Marcinkievicz-type discretization theorem for a general domain $\Omega$ instead of a discrete set $\Omega_M$. Also, in Theorem \ref{T5.2} we have an extra $\log N$ instead of $\log\frac{Nt^2}{\ep^2}$ in (\ref{5.1a}). 

Second, we demonstrate other way of proof, which allows us to replace condition 
{\bf D} by the following more general condition {\bf E}, which is similar to (\ref{5.1}).

{\bf E.} There exists a constant $t$ such that  
\be\label{5.2b}
w(x):=\sum_{i=1}^N u_i(x)^2 \le Nt^2.
\ee

The new way of proof uses the fact that the matrix $G(x)$ is a semi-definite matrix. It is based on the following result (see \cite{Tro12}, Theorem 1.1) on random matrices.

\begin{Theorem}\label{T5.3} Consider a finite sequence $\{T_k\}_{k=1}^m$ of independent, random, self-adjoint matrices with dimension $N$. Assume that 
each random matrix is semi-positive and satisfies 
$$
\lambda_{\max}(T_k) \le R\quad  \text{almost surely}.
$$
Define
$$
s_{\min} := \lambda_{\min}\left(\sum_{k=1}^m \bE(T_k)\right) \quad \text{and}\quad 
s_{\max} := \lambda_{\max}\left(\sum_{k=1}^m \bE(T_k)\right).
$$
Then
$$
\bP\left\{\lambda_{\min}\left(\sum_{k=1}^m T_k\right) \le (1-\eta)s_{\min}\right\} \le 
N\left(\frac{e^{-\eta}}{(1-\eta)^{1-\eta}}\right)^{s_{\min}/R}
$$
for $\eta\in[0,1)$ and for $\eta\ge 0$
$$
\bP\left\{\lambda_{\max}\left(\sum_{k=1}^m T_k\right) \ge (1+\eta)s_{\max}\right\} \le 
N\left(\frac{e^{\eta}}{(1+\eta)^{1+\eta}}\right)^{s_{\max}/R}.
$$
\end{Theorem}

As above,   we consider the matrix $G(x) := [u_i(x)u_j(x)]_{i,j=1}^N$. Clearly, $G(x)$ is a symmetric matrix. Consider a sequence $T_k := G(x^k)$, $k=1,\dots,m$ of independent 
 random symmetric matrices. It is easy to see that $T_k$ are semi-positive definite. Orthonormality of the system $\{u_i\}_{i=1}^N$ implies that $\bE(T_k)=I$ for all $k$. This implies that $s_{\min}=s_{\max} = m$. Relation (\ref{5.2}) shows that we can take $R:=Nt^2$. Then Theorem \ref{T5.3} implies for $\eta\le 1$
 \be\label{5.8a}
\bP\left\{\left\|\sum_{k=1}^m (G(x^k)-I) \right\|\ge m\eta\right\} \le N\exp\left(-\frac{m\eta^2}{ct^2N}\right)
\ee
with an absolute constant $c$ (we can take $c=2/\ln 2$). 
Using inequality (\ref{5.8'}), which was used in the above proof of Theorem \ref{T5.2}, we derive from here the following theorem.

\begin{Theorem}\label{T5.4} Let $\{u_i\}_{i=1}^N$ be an orthonormal system, satisfying condition {\bf E}. Then for every $\ep>0$ there exists a set $\{\xi^j\}_{j=1}^m \subset \Omega$ with
$$
m  \le C\frac{t^2}{\ep^2}N\log N
$$
such that for any $f=\sum_{i=1}^N c_iu_i$ we have
$$
(1-\ep)\|f\|_2^2 \le \frac{1}{m} \sum_{j=1}^m f(\xi^j)^2 \le (1+\ep)\|f\|_2^2.
$$
\end{Theorem}

We note that Theorem \ref{T5.4} is more general and slightly stronger than the Rudelson theorem discussed in the beginning of this section. Theorem \ref{T5.4} provides the Marcinkievicz-type discretization theorem for a general domain $\Omega$ instead of a discrete set $\Omega_M$. Also, in Theorem \ref{T5.4} we have an extra $\log N$ instead of $\log\frac{Nt^2}{\ep^2}$ in (\ref{5.1a}). 

{\bf The Marcinkiewicz theorem and sparse approximation.} Our above argument, in particular, inequality (\ref{5.8'}), shows that the Marcinkiewicz-type 
discretization theorem in $L_2$ is closely related with approximation of the identity matrix $I$ by an $m$-term approximant of the form $\frac{1}{m}\sum_{k=1}^m G(\xi^k)$ in the operator norm from $\ell^N_2$ to $\ell^N_2$ (spectral norm).
Therefore, we can consider the following sparse approximation problem. 
Assume that the system $\{u_i(x)\}_{i=1}^N$ satisfies (\ref{5.2b}) and consider 
the dictionary
$$
\D^u := \{g_x\}_{x\in\Omega},\quad g_x:= G(x)(Nt^2)^{-1},\quad G(x):=[u_i(x)u_j(x)]_{i,j=1}^N.
$$
Then condition (\ref{5.2b}) guarantees that for the Frobenius norm of $g_x$ we have
\be\label{5.11}
\|g_x\|_F = w(x)(Nt^2)^{-1} \le 1.
\ee
Our assumption on the orthonormality of the system $\{u_i\}_{i=1}^N$ gives  
$$
I = \int_\Omega G(x)d\mu = Nt^2\int_\Omega g_x d\mu,
$$
which implies that $I \in A_1(\D^u,Nt^2)$. Consider the Hilbert space 
$H$ to be a closure in the Frobenius norm of  $\sp\{g_x, x\in\Omega\}$ with the inner product generated by the Frobenius norm: for $A=[a_{i,j}]_{i,j=1}^N$ and 
$B=[b_{i,j}]_{i,j=1}^N$ 
$$
\<A,B\> = \sum_{i,j=1}^N a_{i,j}b_{i,j}
$$
in case of real matrices (with standard modification in case of complex matrices). 

By Theorem \ref{T4.1} for any $m\in \N$ we constructively find (by the RGA) 
points $\xi^1,\dots,\xi^m$ such that 
\be\label{5.12}
\left\|\frac{1}{m}\sum_{k=1}^m G(\xi^k)-I\right\|_F \le 2Nt^2 m^{-1/2}.
\ee
Taking into account the inequality $\|A\|\le \|A\|_F$ we get from here and from (\ref{5.8'})
the following proposition.

\begin{Proposition}\label{P5.1} Let $\{u_i\}_{i=1}^N$ be an orthonormal system, satisfying condition {\bf E}. Then there exists a constructive set $\{\xi^j\}_{j=1}^m \subset \Omega$ with $m\le C(t)N^2$
such that for any $f=\sum_{i=1}^N c_iu_i$ we have
$$
\frac{1}{2}\|f\|_2^2 \le \frac{1}{m} \sum_{j=1}^m f(\xi^j)^2 \le \frac{3}{2}\|f\|_2^2.
$$
\end{Proposition}

{\bf The Marcinkiewicz-type theorem with weights.} We now comment on a recent breakthrough result by J. Batson, D.A. Spielman, and N. Srivastava \cite{BSS}. We formulate their result in our notations. Let as above $\Omega_M=\{x^j\}_{j=1}^M$ be a discrete set with the probability measure $\mu(x^j)=1/M$, $j=1,\dots,M$. Assume that 
$\{u_i(x)\}_{i=1}^N$ is a real orthonormal on $\Omega_M$ system. Then for any 
number $d>1$ there exist a set of weights $w_j\ge 0$ such that $|\{j: w_j\neq 0\}| \le dN$ so that for any $f\in \sp\{u_1,\dots,u_N\}$ we have
$$
\|f\|_2^2 \le \sum_{j=1}^M w_jf(x^j)^2 \le \frac{d+1+2\sqrt{d}}{d+1-2\sqrt{d}}\|f\|_2^2.
$$
The proof of this result is based on a delicate study of the $m$-term approximation of the identity matrix $I$ with respect to the system 
$\D := \{G(x)\}_{x\in \Omega}$, $G(x):=[u_i(x)u_j(x)]_{i,j=1}^N$ in the spectral norm. The authors of \cite{BSS} control the change of the maximal and minimal eigenvalues of a matrix, when they add a rank one matrix of the form $wG(x)$. 
Their proof provides an algorithm for construction of the weights $\{w_j\}$. 
In particular, this implies that 
$$
X_N(\Omega_M) \in \cM^w(m,2,\epsilon)\quad \text{provided} \quad m \ge CN\epsilon^{-2}
$$
with large enough $C$.

In this section we discussed two deep general results -- the Rudelson theorem and the Batson-Spielman-Srivastava theorem -- about submatrices of orthogonal matrices, which 
provide very good Marcinkiewicz-type discretization theorems for $L_2$.
The reader can find a corresponding historical comments in 
\cite{Rud}. We also refer the reader to the paper \cite{Ka} for a discussion of a
recent outstanding progress on the theory of submatrices of orthogonal matrices. 

\section{Discussion}

As we pointed out in the Introduction the main results of this paper are on the 
Marcinkiewicz-type discretization theorems in $L_1$. We proved here that under certain conditions on a subspace $X_N$ we can get the corresponding discretization theorems with the number of knots $m\ll N(\log N)^{7/2}$. This result is only away from the ideal case $m=N$ by the $(\log N)^{7/2}$ factor. 
We point out that the situation with the discretization theorems in the $L_\infty$ case is fundamentally different. A very nontrivial surprising negative result was proved for the $L_\infty$ case (see \cite{KT3}, \cite{KT4}, and \cite{KaTe03}). The authors proved that the necessary condition for
$\Tr(Q_n)\in\cM(m,\infty)$ is $m\gg |Q_n|^{1+c}$ with absolute constant $c>0$.

Theorem \ref{T4.10} shows that an important ingredient of our technique of proving the Marcinkiewicz discretization theorems in $L_1$ consists in the study of the entropy numbers $\e_k(X^1_N,L_\infty)$. We note that this problem is a nontrivial problem by itself. We demonstrate this on the example of the trigonometric polynomials. It is proved in \cite{VT156} that in the case $d=2$
we have
\be\label{6.1}
\e_k(\Tr( Q_n)_1,L_\infty)\ll  n^{1/2} \left\{\begin{array}{ll} (| Q_n|/k) \log (4| Q_n|/k), &\quad k\le 2| Q_n|,\\
2^{-k/(2| Q_n|)},&\quad k\ge 2| Q_n|.\end{array} \right.
\ee
The proof of estimate (\ref{6.1}) is based on an analog of the Small Ball Inequality for the trigonometric system proved for the wavelet type system (see \cite{VT156}). This proof uses the two-dimensional specific features of the problem and we do not know how to extend this proof to the case $d>2$. 
Estimate (\ref{6.1}) is used in the proof of the upper bounds in Theorem \ref{T2.7}. Theorem \ref{T2.7} gives the right order of the entropy numbers for 
the classes of mixed smoothness. This means that (\ref{6.1}) cannot be substantially improved. The trivial inequality $\log (4| Q_n|/k) \ll n$ shows that 
(\ref{6.1}) implies the following estimate
\be\label{6.2}
\e_k(\Tr( Q_n)_1,L_\infty)\ll  n^{3/2} \left\{\begin{array}{ll} | Q_n|/k , &\quad k\le 2| Q_n|,\\
2^{-k/(2| Q_n|)},&\quad k\ge 2| Q_n|.\end{array} \right.
\ee
Estimate (\ref{6.2}) is not as good as (\ref{6.1}) in application for proving the upper bounds of the entropy numbers of smoothness classes. For instance,
instead of the bound in Theorem \ref{T2.7} use of (\ref{6.2}) will give
$$
\e_k(\bW^{a,b}_1,L_\infty) \ll k^{-a}(\log k)^{a+b+3/2}.
$$
However, it turns out that in application to the Marcinkiewicz-type discretization theorems estimates (\ref{6.1}) and (\ref{6.2}) give the same bounds on the number of knots $m\ll |Q_n|n^{7/2}$ (see Theorem \ref{T3.1} and Theorem \ref{T3.3}). 

As we pointed out above we do not have an extension of (\ref{6.1}) to the case 
$d>2$. A somewhat straight forward technique presented in \cite{VT160} gives the following result for all $d$
\be\label{6.3}
\e_k(\Tr( Q_n)_1,L_\infty)\ll  n^{d/2} \left\{\begin{array}{ll} (| Q_n|/k) \log (4| Q_n|/k), &\quad k\le 2| Q_n|,\\
2^{-k/(2| Q_n|)},&\quad k\ge 2| Q_n|.\end{array} \right.
\ee
This result is used in \cite{VT160} to prove Theorem \ref{T3.2}. An interesting 
contribution of this paper is the proof of (\ref{6.2}) for all $d$ and for  rather general sets $\Tr(Q)_1$ instead of $\Tr(Q_n)_1$. An important new ingredient here is the use of dictionary $\D^2(Q)$, consisting of shifts of normalized Dirichlet kernels associated with $Q$, in $m$-term approximations. 
Certainly, it would be nice to understand, even in the special case of the hyperbolic cross polynomials $\Tr(Q_n)$, if the embedding $\Tr(Q_n) \in \cM(m,1)$ with $m\asymp |Q_n|$ holds. Results of this paper only show that the above embedding holds with $m \gg |Q_n|n^{7/2}$. We got the extra factor $n^{7/2}$ as a result of using (\ref{6.2}), which contributed $n^{3/2}$, and of using the chaining technique, which contributed $n^2$.

\end{document}